\documentclass[draft,reqno]{amsproc}
\usepackage{amsmath}
\usepackage{amsfonts}
\usepackage{mathrsfs}
\usepackage[dvips]{graphicx}
\usepackage{amssymb}
\usepackage{amsbsy}
\usepackage{amsthm}
\usepackage{euscript}   %\EuScript{A}
\usepackage{graphics}
\usepackage{color}
\usepackage{textcomp}
\usepackage{verbatim}

\makeatletter
\@namedef{subjclassname@2010}{%
\textup{2010} Mathematics Subject
Classification} \makeatother

\numberwithin{equation}{section}

\newtheorem{theorem}{Theorem}[section]
\newtheorem*{theorem*}{Theorem}
\newtheorem{corollary}[theorem]{Corollary}
\newtheorem{prop}[theorem]{Proposition}

\newtheorem*{prob*}{Problem}

\theoremstyle{remark}
\newtheorem{rem}[theorem]{Remark}

\theoremstyle{definition}
\newtheorem{exa}[theorem]{Example}

\newcommand*{\ascr}{\mathscr{A}}
\newcommand*{\bscr}{\mathscr{B}}

\newcommand*{\gammab}{\boldsymbol{\gamma}}
\newcommand*{\hh}{\mathcal{H}}

\newcommand*{\kk}{\mathcal{K}}

\newcommand*{\natu}{\mathbb{N}}

\newcommand*{\borel}[1]{\mathfrak{B}(#1)}

\newcommand*{\D}{\mathrm{d\hspace{.1ex}}}

\newcommand*{\Le}{\leqslant}
\newcommand*{\Ge}{\geqslant}
\newcommand*{\ran}{\mathscr{R}}

\newcommand*{\ogr}[1]{\boldsymbol{B}(#1)}
\newcommand*{\rbb}{\mathbb{R}}
\newcommand*{\zbb}{\mathbb{Z}}

\begin{document}

\dedicatory{In memory of Professor Franciszek
Hugon Szafraniec, our teacher and mentor, whose
vision and commitment left a lasting mark on our
mathematical community.}

   \title[ CPD $n$th roots of subnormal operators are subnormal]
{CPD $n$th roots of subnormal operators are subnormal}

    \author[Z.J. Jab{\l}o{\'n}ski, I. B. Jung, P. Pietrzycki and J. Stochel]{ Zenon Jan Jab{\l}o{\'n}ski,  Il Bong Jung, Pawe{\l} Pietrzycki and  Jan Stochel}

   \subjclass[2020]{}

   \address{Wydzia{\l} Matematyki i Informatyki, Uniwersytet
Jagiello\'{n}ski, ul. {\L}ojasiewicza 6,
PL-30348 Krak\'{o}w, Poland}

   \email{zenon.jablonski@im.uj.edu.pl}

   \address{Department of Mathematics, Kyungpook National University, Daegu 702-701, Ko
rea}

   \email{ibjung@knu.ac.kr}
   \address{Wydzia{\l} Matematyki i Informatyki, Uniwersytet
Jagiello\'{n}ski, ul. {\L}ojasiewicza 6,
PL-30348 Krak\'{o}w, Poland}

   \email{pawel.pietrzycki@im.uj.edu.pl}

   \address{Wydzia{\l} Matematyki i Informatyki, Uniwersytet
Jagiello\'{n}ski, ul. {\L}ojasiewicza 6,
PL-30348 Krak\'{o}w, Poland}

   \email{jan.stochel@im.uj.edu.pl}


    \thanks{This work was supported by the National Science Centre, Poland, under
project OPUS grant no. DEC-2024/55/B/ST1/00280.}

   \begin{abstract}
We investigate the $n$th-root problem for bounded
operators on a Hilbert space within the class of
conditionally positive definite (CPD) operators
determined by the L\'evy--Khintchine formula.
This class contains subnormal operators, complete
hypercontractions of order $2$, and
$3$-isometries. Our main result shows that if $T$
is a CPD operator such that $T^n$ is subnormal
(respectively, quasinormal, normal, or a
$3$-isometry), then $T$ belongs to the
corresponding class. This establishes that these
classes are invariant under taking $n$th roots
within the CPD class and extends several earlier
results in operator theory. Furthermore, we
characterize quasinormal and normal operators in
terms of the CPD property and the structure of
the associated representing triplet. Finally,
using both theoretical arguments and explicit
examples, we show that the classes of CPD and
normaloid operators are distinct.
   \end{abstract}

   \maketitle

   \section{Introduction}
The investigation of $n$th roots of operators
within selected classes has a decades-long
history, with origins dating back at least to the
early 1950s; see, for example,
\cite{hal53,hal54,put57,Emb68,Ra-Ros,Keo84,wog85,con87,Dug93,Il-Ku19,Ki-Ko22,M-P-R23,Stan23,Stankovic2025,StankovicKubrusly2025,CurtoPrasad2026}.
J. G. Stampfli proved that a hyponormal $n$th
root of a normal operator is normal (see
\cite[Theorem~5]{sta62}). T.~Ando improved this
result showing that a paranormal $n$th root of a
normal operator is normal (see
\cite[Theorem~6]{Ando72}). However, a hyponormal
$n$th root of a subnormal operator need not be
subnormal (see \cite[pp.\ 378/379]{sta66}). On
the other hand, normal operators and non-normal
quasinormal operators can have non-normal and
non-quasinormal $n$th roots, respectively (see
\cite[Section~6]{P-S23}).

The present work can be regarded as a continuation of the article
\cite{P-S21}, in which the authors solved affirmatively \cite[Problem
1.1]{Curto20}, posed by R. E. Curto, S. H. Lee, and J. Yoon, by
proving a more general result stated below. Here, $\ogr{\hh}$ denotes
the $C^*$-algebra of all bounded (linear) operators on a (complex)
Hilbert space $\hh$.
   \begin{theorem}[{\cite[Theorem~1.2]{P-S21}}] \label{twPS}
If $T\in \ogr{\hh}$ is a subnormal operator such that $T^n$ is
quasinormal, where $n$ is a positive integer, then $T$ is
quasinormal.
   \end{theorem}
This result was generalized in \cite{P-S23} from subnormal operators
to operators of class $A$. The operator measure techniques used in
the proof of Theorem~\ref{twPS} are closely connected with Arveson's
hyperrigidity conjecture (see \cite[Theorem~4.2]{P-S21} and
\cite[Lemma~8.1]{P-S-24}).

The celebrated characterization of subnormality due to Lambert states
that a bounded operator $T$ on a Hilbert space $\hh$ is subnormal if
and only if, for every $h \in \hh$, the sequence $\{\|T^n
h\|^2\}_{n=0}^{\infty}$ is positive definite as a function on the
additive semigroup of all nonnegative integers $\zbb_+$ (see
\cite{Lam76}; see also \cite[Theorem~7]{Sto-Szaf89}). In the harmonic
analysis on $*$-semigroups developed in \cite{B-C-R}, conditionally
positive definite functions emerge as fundamental objects. Inspired
by this, conditionally positive definite (CPD) operators were
introduced and studied in \cite{J-J-S22}. An operator $T$ is said to
be CPD if, for every $h \in \hh$, the sequence $\{\|T^n
h\|^2\}_{n=0}^{\infty}$ is conditionally positive definite as a
function on the semigroup $\zbb_+$. The class of CPD operators
contains, in particular, subnormal operators \cite{hal50,Con91},
complete hypercontractions of order $2$ \cite{Cha-Sh15},
$3$-isometries \cite{Ag-St1,Ag-St2,Ag-St3}, and many others.

The central aim of the $n$th-root problem is to
determine which classes of operators, say
$\EuScript{A}$, are invariant under taking $n$th
roots within a prescribed class of operators, say
$\EuScript{B}$, where
$\EuScript{A}\subseteq\EuScript{B}$. In this
paper, we focus on the case where $\EuScript{B}$
is the class of CPD operators, while
$\EuScript{A}$ ranges over the classes of
subnormal, quasinormal, normal, and $3$-isometric
operators. Our main result in this direction is
as follows.
   \begin{theorem}  \label{twes}
Let $n \Ge 2$ be an integer, and let $T \in \ogr{\mathcal{H}}$ be a
CPD operator such that $T^n$ is subnormal $($resp., quasinormal,
normal, or a $3$-isometry$)$. Then $T$ is subnormal $($resp.,
quasinormal, normal, or a $3$-isometry$)$.
   \end{theorem}
The proof of Theorem~\ref{twes} is divided into four parts,
corresponding to the four cases in the theorem, and is carried out in
Section~\ref{Sec.4} (Theorem~\ref{thp1} and Corollary~\ref{thpc}) and
Section~\ref{Sec.5} (Theorems~\ref{thp5} and~\ref{thp52}). In
Section~\ref{Sec.3}, we provide characterizations of quasinormal and
normal operators in terms of their CPD property and the structure of
the representing triplet (see Theorem~\ref{thp3} and
Proposition~\ref{thp3b}). In Section~\ref{Sec.6}, we show, both by a
purely theoretical argument and by concrete examples, that the
classes of CPD operators and normaloid operators are distinct, in the
sense that neither is contained in the other.
   \section{Preliminaries}
In this paper, we use the following notation. The fields of real and
complex numbers are denoted by $\mathbb{R}$ and $\mathbb{C}$,
respectively. The symbols $\mathbb{Z}_+$, $\mathbb{N}$, and
$\mathbb{R}_+$ stand for the sets of nonnegative integers, positive
integers, and nonnegative real numbers, respectively. Given a set
$\varDelta \subseteq \mathbb{C}$, we write $\varDelta^* = \{\bar{z}
\colon z \in \varDelta\}$. We denote by $\borel{\varOmega}$ the
$\sigma$-algebra of all Borel subsets of a Hausdorff topological
space $\varOmega$.

A sequence $\gammab=\{\gamma_n\}_{n=0}^{\infty} \subseteq \mathbb{R}$
is said to be \emph{positive definite} (\emph{PD} for brevity) if
   \begin{align} \label{virek}
\sum_{i,j=0}^k \gamma_{i+j} \lambda_i \bar{\lambda}_j
\Ge 0
   \end{align}
for all finite sequences $\{\lambda_i\}_{i=0}^{k} \subseteq
\mathbb{C}$. The sequence $\gammab$ is said to be \emph{conditionally
positive definite} (\emph{CPD}) if \eqref{virek} holds for all finite
sequences $\{\lambda_i\}_{i=0}^{k} \subseteq \mathbb{C}$ such that
$\sum_{i=0}^k \lambda_i = 0$. Trivially, every PD sequence is CPD.

A sequence $\gammab=\{\gamma_n\}_{n=0}^{\infty}
\subseteq \mathbb{R}$ is said to be a {\em Stieltjes
moment sequence} if there exists a positive Borel
measure $\mu$ on $\mathbb{R}_+$ such that
   \begin{align} \label{hamb}
\gamma_n = \int_{\mathbb{R}_+} t^n \, d\mu(t), \quad n
\in \mathbb{Z}_+.
   \end{align}
A positive Borel measure $\mu$ on $\mathbb{R}_+$ satisfying
\eqref{hamb} is called a {\em representing measure} of $\gammab$. The
Stieltjes theorem states that $\gammab$ is a Stieltjes moment
sequence if and only if the sequences $\gammab$ and
$\{\gamma_{n+1}\}_{n=0}^{\infty}$ are PD (see
\cite[Theorem~6.2.5]{B-C-R}). If $\gammab$ is a Stieltjes moment
sequence which has a unique representing measure, then we say that
$\gammab$ is {\em determinate}. It is well known that if a Stieltjes
moment sequence has a representing measure with compact support, then
it is determinate. The reader is referred to \cite{B-C-R} for the
foundations of the theory of moment problems.

The $C^*$-algebra of all bounded linear operators on a (complex)
Hilbert space $\hh$ is denoted by $\ogr{\hh}$. The identity operator
on $\hh$ is denoted by $I$. Let $\ascr$ be a $\sigma$-algebra of
subsets of a set $\varOmega$, and let $F\colon \ascr \to \ogr{\hh}$
be a {\em semispectral measure}, that is, $\langle F(\cdot)f,
f\rangle$ is a positive measure for every $f \in \hh$. Denote by
$L^1(F)$ the vector space of all $\ascr$-measurable functions
$f\colon \varOmega \to \mathbb{C}$ such~that
\[
\int_{\varOmega} |f(x)| \, \langle F(\D x)h, h\rangle < \infty \quad
\text{for all } h \in \hh.
\]
Then, for every $f \in L^1(F)$, there exists a unique operator
$\int_\varOmega f \, \D F \in \ogr{\hh}$ such that (see, e.g.,\
\cite[Appendix]{Sto92})
\begin{align*}
\Big\langle \int_\varOmega f \, \D F\, h, h \Big\rangle =
\int_\varOmega f(x) \, \langle F(\D x)h, h\rangle, \quad h \in \hh.
\end{align*}
In this paper, we do not assume that a semispectral measure $F$ on
$\ascr$ is normalized, that is, that $F(\varOmega) = I$. Observe that
if $F$ is a {\em spectral measure}, that is, a normalized
semispectral measure such that $F(\varDelta)$ is an orthogonal
projection for every $\varDelta \in \ascr$, then $\int_\varOmega f \,
\D F$ coincides with the usual spectral integral. If $F$ is the
spectral measure of a normal operator $T$, then we write
$f(T)=\int_{\mathbb{C}} f \, \D F$ for a Borel function $f\colon
\mathbb{C} \to \mathbb{C}$; the map $f \mapsto f(T)$ is called the
Stone--von Neumann functional calculus. We refer the reader to
\cite{Weid80,Bir-Sol87,Sch12} for the basic facts of spectral theory.

An operator $T \in \ogr{\hh}$ is said to be {\em subnormal} if there
exists a Hilbert space $\kk$ containing $\hh$ and a normal operator
$N \in \ogr{\kk}$ that extends $T$. By a {\em semispectral measure}
of a subnormal operator $T \in \ogr{\hh}$ we mean a normalized,
compactly supported semispectral measure $G \colon \borel{\mathbb{C}}
\to \ogr{\hh}$ defined by $G(\varDelta) = P E(\varDelta)|_{\hh}$ for
$\varDelta \in \borel{\mathbb{C}}$, where $E \colon
\borel{\mathbb{C}} \to \ogr{\kk}$ is the spectral measure of a
minimal normal extension $N \in \ogr{\kk}$ of $T$, and $P \in
\ogr{\kk}$ is the orthogonal projection of $\kk$ onto $\hh$
(minimality means that $\kk$ has no proper closed subspace that
reduces $N$ and contains $\hh$). It is easily seen that
\begin{align*}
T^{*m} T^n = \int_{\mathbb{C}} \bar z^m z^n \, G(\D z), \quad m,n \in
\mathbb{Z}_+.
\end{align*}
Hence, applying the measure transport theorem (cf.\
\cite[Theorem~1.6.12]{Ash00}), we obtain
\begin{align}  \label{tobemom}
T^{*n} T^n = \int_{\mathbb{R}_+} x^n \, G \circ \phi^{-1}(\D x),
\quad n \in \mathbb{Z}_+,
\end{align}
where $\phi \colon \mathbb{C} \to \mathbb{R}_+$ is defined by
$\phi(z)=|z|^2$ for $z \in \mathbb{C}$, and $G \circ \phi^{-1} \colon
\borel{\mathbb{R}_+} \to \ogr{\hh}$ is the normalized semispectral
measure defined by $G \circ \phi^{-1}(\varDelta) =
G(\phi^{-1}(\varDelta))$ for $\varDelta \in \borel{\mathbb{R}_+}$. By
Lambert's theorem, an operator $T \in \ogr{\hh}$ is subnormal if and
only if $\{\|T^n f\|^2\}_{n=0}^{\infty}$ is PD for all $f \in \hh$
(see \cite{Lam76}). The foundations of the theory of subnormal
operators can be found in \cite{Con91}.

Let $m\in\mathbb{N}$. An operator
$T\in\boldsymbol{B}(\mathcal{H})$ is said to be
an \emph{$m$-isometry} if the equality
$\bscr_m(T)=0$ holds, where
   \begin{align*}
\bscr_m(T) :=
\sum_{j=0}^{m}(-1)^{m-j}\binom{m}{j}T^{*j}T^j.
   \end{align*}
A systematic study of $m$-isometries was
initiated by Agler and Stankus in their trilogy
\cite{Ag-St1,Ag-St2,Ag-St3}.

Following \cite{J-J-S22}, we say that an operator
$T \in \ogr{\hh})$ is \emph{conditionally
positive definite} (CPD) if, for every $h \in
\mathcal{H}$, the sequence $\{\|T^n
h\|^2\}_{n=0}^{\infty}$ is CPD. By Lambert's
theorem, subnormal operators are CPD. The
following theorem characterizes CPD operators in
terms of triplets consisting of two operators and
a semispectral measure. We call the identity
\eqref{cdr5} the L\'evy--Khintchine formula for a
CPD operator~$T$.
   \begin{theorem}[{\cite[Theorem~3.1.1]{J-J-S22}}] \label{cpdops} Let
$T\in \ogr{\hh}$. Then the following statements are equivalent{\em :}
   \begin{enumerate}
   \item[\rm(i)] $T$ is CPD,
   \item[\rm(ii)] there exist operators $B,C\in
\ogr{\hh}$ and a compactly supported semispectral measure $F\colon
\borel{\rbb_+} \to \ogr{\hh}$ such that $B=B^*$, $C\Ge 0$,
$F(\{1\})=0$ and
   \begin{align} \label{cdr5}
T^{*n}T^n = I + n B + n^2 C + \int_{\rbb_+} Q_n \, \D F, \quad n\in
\zbb_+,
   \end{align}
where $Q_n$ is a polynomial in $x$ defined by
   \begin{align} \notag%\label{klaud}
Q_n(x) =
   \begin{cases}
0 & \text{if } x \in \rbb \text{ and } n=0,1,
   \\[1ex]
\sum_{j=0}^{n-2} (n -j -1) x^j & \text{if } x \in \rbb \text{ and }
n\Ge 2.
   \end{cases}
   \end{align}
   \end{enumerate}
Moreover, if {\em (ii)} holds, then the triplet $(B,C,F)$ is unique.
%and
%   \begin{gather} \label{fontan}
%\supp{F} \subseteq [0,r(T)^2],
%   \\ \label{fontan5}
%C\neq 0 \implies r(T) \Ge 1,
%   \\ \label{fontan6}
%\sup \supp{F} \Ge 1 \implies r(T)^2 =
%\sup \supp{F}.
%   \end{gather}
%Furthermore, $(\is{Bh}h, \is{Ch}h,
%\is{F(\cdot)h}h)$ is the representing
%triplet of the CPD sequence $\{\|T^n
%h\|^2\}_{n=0}^{\infty}$ for every $h\in
%\hh$.
   \end{theorem}
The unique triple \((B,C,F)\) in Theorem~\ref{cpdops} is called the
\emph{representing triple} of~\(T\). The polynomials $Q_n$ can also
be described as follows:
   \begin{align} \label{rnx-1}
Q_n(x) & = \frac{x^n-1 - n (x-1)}{(x-1)^2}, \quad x\in \rbb\setminus
\{1\}, \, n \in \zbb_+.
   \end{align}

There is also another way to describe CPD operators using a single
object, namely a semispectral measure, denoted by $M$.
   \begin{theorem}[{\cite[Theorems~3.2.5 and 3.3.1]{J-J-S22}}]
\label{dyl-an} For $T \in \ogr{\hh}$, the following conditions are
equivalent:
      \begin{enumerate}
   \item[(i)] $T$ is CPD,
   \item[(ii)] there exists a semispectral measure
$M \colon \borel{\mathbb{R}_+} \to \ogr{\hh}$ with compact support
such that
   \begin{align} \notag
T^{*n}\bscr_2(T)T^n = \int_{\mathbb{R}_+} x^n \,
M(\D x), \quad n \in \mathbb{Z}_+.
   \end{align}
   \end{enumerate}
Moreover, if {\em (ii)} holds, then the semispectral measure $M$ is
unique, and if $(B, C, F)$ is a representing triplet of $T$, then
   \begin{gather} \label{f2m-remi}
B + C = T^*T - I, \qquad C = \frac{1}{2} M(\{1\}),
   \\ \label{f2m-semi}
F(\varDelta) = \bigl(1 - \chi_{\varDelta}(1)\bigr) M(\varDelta),
\quad \varDelta \in \borel{\mathbb{R}_+}.
   \end{gather}
   \end{theorem}
The following theorem provides a characterization of subnormal
operators in terms of the CPD property and the structure of the
representing triplet, and will serve as a useful reference point in
what follows.
   \begin{theorem}[{\cite[Theorem~3.4.1]{J-J-S22}}] \label{subn-1}
Let $T \in \ogr{\hh}$. Then the following statements are equivalent:
   \begin{enumerate}
   \item[(i)] $T$ is subnormal,
   \item[(ii)] $T$ is CPD and its representing triplet $(B, C, F)$
satisfies the following conditions:
   \begin{enumerate}
   \item[(a)] $\frac{1}{(x-1)^2} \in L^1(F)$ and
$\int_{\mathbb{R}_+} \frac{1}{(x-1)^2} \, F(\D x) \Le
I$,
   \item[(b)] $\frac{1}{x-1} \in L^1(F)$ and
$B = \int_{\mathbb{R}_+} \frac{1}{x-1} \, F(\D x)$,
   \item[(c)] $C = 0$.
   \end{enumerate}
   \end{enumerate}
Moreover, if {\em (ii)} holds and $G$ is the semispectral measure of
$T$ $($see \eqref{tobemom}$)$, then \allowdisplaybreaks
   \begin{align} \label{feme}
F & = M, \text{ where $M$ is as in Theorem~{\em \ref{dyl-an}(ii)}},
   \\ \notag
G \circ \phi^{-1}(\varDelta) &= \int_{\varDelta} \frac{1}{(x-1)^2} \,
F(\D x)
   \\ \label{ddwj}
& \hspace{5ex} + \delta_1(\varDelta) \Bigl(I - \int_{\mathbb{R}_+}
\frac{1}{(x-1)^2} \, F(\D x)\Bigr), \quad \varDelta \in
\borel{\mathbb{R}_+}.
   \end{align}
   \end{theorem}
   \section{\label{Sec.3}Quasinormal operators as CPD operators}
The following theorem provides
equivalent characterizations of
quasinormal operators in terms of their
CPD property and the structure of the
representing triplet, emphasizing that,
since quasinormal operators are
subnormal (see
\cite[Proposition~II.1.7]{Con91}), they
naturally share certain conditions that
characterize both classes (see
Theorem~\ref{subn-1}). Recall that an
operator $T\in \ogr{\hh}$ is called
{\em quasinormal} if it commutes with
$T^*T$ (see \cite{brow53}).
   \begin{theorem}\label{thp3}
Let $T \in \ogr{\hh}$. Then the
following statements are equivalent{\em
:}
\begin{itemize}
    \item[(i)] $T$ is quasinormal,
    \item[(ii)] $T$ is CPD and its representing triplet
$(B, C, F)$ satisfies the following{\em
:}
   \begin{enumerate}
   \item[(a)] $\frac{1}{(x-1)^2} \in L^1(F)$
and the map $\borel{\mathbb{R}_+} \ni
\varDelta \to \int_{\varDelta}
\frac{1}{(x-1)^2} F(\D x) \in
\ogr{\hh}$ is an orthogonal
projection-valued semispectral measure,
   \item[(b)] $\frac{1}{x-1} \in L^1(F)$ and
$B = \int_{\mathbb{R}_+} \frac{1}{x-1}
F(\D x)$,
   \item[(c)] $C=0$.
   \end{enumerate}
   \item[(iii)] $T$ is CPD and its representing triplet
$(B, C, F)$ satisfies {\em (b)}, {\em
(c)}, and the following condition{\em
:}
   \begin{enumerate}
   \item[(a$^*$)] there exists $P\colon \borel{\rbb_+}
\to \ogr{\hh}$, an orthogonal
projection-valued semispectral measure
with compact support, such that
   \begin{align} \label{fdod}
F(\varDelta)=\int_{\varDelta} (x-1)^2 P(\D x), \quad \varDelta \in
\borel{\rbb_+}.
   \end{align}
   \end{enumerate}
   \end{itemize}
Moreover, if $T$ is quasinormal, and
$P$ is as in {\em (a$^*$)}, then
   \begin{align} \label{pfrc}
P(\varDelta) = \int_{\varDelta} \frac{1}{(x-1)^2}
F(\D x) = G\circ\phi^{-1}(\varDelta), \quad
\varDelta \in \borel{\rbb_+ \setminus \{1\}}.
   \end{align}
   \end{theorem}
   \begin{proof}
    (i)$\Rightarrow$(ii) Since every quasinormal
operator is subnormal, we may apply
Theorem~\ref{subn-1}. It follows from Embry's
characterization of quasinormal operators (see
\cite[p.\ 63]{Embry73}) that
   \begin{align*} %\label{irct}
\int_{\mathbb{R}_+} x^n
G\circ\phi^{-1}(\D x)
\overset{\eqref{tobemom}}=T^{*n} T^{n}
= (T^{*} T)^{n} = \int_{\mathbb{R}_+}
x^n E_{T^*T}(\D x), \quad n\in \zbb_+,
   \end{align*}
where $E_{T^*T}$ is the spectral measure of $T^*T$.
Since the measures $\langle G\circ
\phi^{-1}(\cdot)h,h\rangle$ and $\langle
E_{T^*T}(\cdot)h,h\rangle$ are compactly supported,
the corresponding Stieltjes moment sequence is
determinate (see \cite[Corollary~4.2]{Sch17}). Hence,
$G\circ\phi^{-1}=E_{T^*T}$, and thus $G\circ\phi^{-1}$
is a spectral measure. Therefore, by \eqref{ddwj}, we
have
   \begin{align*}
\int_{\varDelta} &\frac{1}{(x-1)^2} F(\D x) =
\int_{\varDelta\setminus \{1\}} \frac{1}{(x-1)^2}
F(\D x) =G\circ\phi^{-1}(\varDelta\setminus
\{1\}),\quad \varDelta \in\borel{\mathbb{R}_+}.
   \end{align*}
It follows that $\int_{\varDelta} \frac{1}{(x-1)^2}\,F(\D x)$ is an
orthogonal projection for every $\varDelta\in\borel{\mathbb{R}_+}$,
and that the second equality in \eqref{pfrc} holds.

(ii)$\Rightarrow$(i) Since $Q:=\int_{\rbb_+} \frac{1}{(x-1)^2} F(\D
x)$ is an orthogonal projection, we deduce that $\int_{\rbb_+}
\frac{1}{(x-1)^2} F(\D x)\Le I$. Hence, by Theorem~\ref{subn-1}, $T$
is subnormal, so \eqref{tobemom} holds. It follows from {\rm (a)}
that the map $P\colon \borel{\rbb_+}\to\ogr{\hh}$, defined by
   \begin{align}  \label{rmev}
P(\varDelta) = \int_{\varDelta} \frac{1}{(x-1)^2} F(\D x), \quad
\varDelta \in \borel{\rbb_+},
   \end{align}
is a compactly supported orthogonal projection-valued semispectral
measure. Since $Q$ is an orthogonal projection, $I - Q$ is the
orthogonal projection onto $\ran(Q)^\perp$, where $\ran(S)$ denotes
the range of an operator $S\in \ogr{\hh}$. Note that $P(\varDelta)
\Le Q$ for every $\varDelta \in\borel{\mathbb{R}_+}$. Therefore,
$\ran(P(\varDelta)) \subseteq \ran(Q)$, which yields
   \begin{align*}
\ran(P(\varDelta)) \perp
\ran(Q)^{\perp} = \ran(I-Q),
   \end{align*}
or, equivalently,
   \begin{align} \label{rpqi}
P(\varDelta) (I-Q) = (I-Q) P(\varDelta)
= 0.
   \end{align}
This, together with \eqref{ddwj}, implies that
   \begin{align} \label{gpiq}
G\circ\phi^{-1}(\varDelta) =
P(\varDelta) + \delta_1(\varDelta)
(I-Q), \quad \varDelta \in
\borel{\mathbb{R}_+}.
   \end{align}
Combining \eqref{rpqi} and
\eqref{gpiq}, we conclude that
$G\circ\phi^{-1}$ is a spectral measure
satisfying \eqref{tobemom}. By Embry's
characterization of quasinormal
operators (see \cite[p.\ 63]{Embry73}),
$T$ is quasinormal.

(ii)$\Rightarrow$(iii) By (a), the map $P\colon
\borel{\rbb_+} \to \ogr{\hh}$, defined by
\eqref{rmev}, is a compactly supported orthogonal
projection-valued semispectral measure. Since
$F(\{1\})=0$, it follows from
\cite[Theorem~1.29]{Rud87} that, for all
$\varDelta \in \borel{\rbb_+}$ and $h \in \hh$,
   \begin{align*}
\left\langle \int_{\varDelta} (x-1)^2 P(\D x) h,h \right \rangle =
\int_{\varDelta} (x-1)^2 \langle P(\D x) h,h \rangle = \langle
F(\varDelta)h,h \rangle.
   \end{align*}
This yields (a$^*$).

(iii)$\Rightarrow$(ii) Applying
\cite[Theorem~1.29]{Rud87} once again, we obtain,
for every $\varDelta\in\borel{\rbb_+}$ and every
$h\in\hh$,
\begin{align*}
\int_{\varDelta}\frac{1}{(x-1)^2} \langle F(\D
x)h,h\rangle \overset{\eqref{fdod}}{=} \langle
P(\varDelta\setminus{1})h,h\rangle.
\end{align*}
This proves (a) and the first equality in
\eqref{pfrc}, thereby completing the proof.
   \end{proof}
Normal operators can be characterized as follows.
   \begin{prop}\label{thp3b}
If $T \in \ogr{\hh}$, then the following statements
are~equivalent{\em :}
\begin{itemize}
    \item[(i)] $T$ is normal,
    \item[(ii)] $T$ and $T^*$ are CPD and have the same representing
triplets.
   \end{itemize}
   \end{prop}
   \begin{proof}
It suffices to observe that, if $T$ and $T^*$ are CPD, then, by the
uniqueness of representing triplets (see Theorem~\ref{cpdops}), the
operators $T$ and $T^*$ have the same representing triplets if and
only if
\[
T^{*n}T^n = T^n T^{*n}, \quad n \in \mathbb{Z}_+,
\]
which, in turn, is equivalent to the normality of $T$.
   \end{proof}
   \begin{rem}
T. Ando proved in \cite[Theorem~5]{Ando72} that an operator $T \in
\ogr{\hh}$ is normal if and only if both $T$ and $T^*$ are
paranormal, and have the same kernels. This result may suggest that,
in Proposition~\ref{thp3b}(ii), it would suffice to repeat Ando's
assumptions with paranormality replaced by conditional positive
definiteness. Unfortunately, Example~4.9 of \cite{Buchala25} provides
a $3$-isometric operator $S$ such that $S^*$ is also a $3$-isometry.
Hence both $S$ and $S^*$ are injective CPD operators, since
$3$-isometries are always injective. Moreover, this particular
operator $S$ is not normal. Indeed, $S$ is a bilateral weighted shift
with $2\times2$ matrix weights $\{S_j\}_{j\in\zbb}$ such that $S_1$
is not unitary. Consequently, $S$ itself is not unitary. If $S$ were
normal, then, by \cite[Proposition~4.5]{Sh-At2000}, it would be an
isometry. Since $S$ has dense range, it would then be unitary, a
contradiction.
   \hfill $\diamondsuit$
   \end{rem}
   \section{\label{Sec.4}Subnormality of CPD $n$th roots}
In this section we show that subnormality is preserved under taking
CPD $n$th roots.
   \begin{theorem}\label{thp1}
Let \( n \Ge 2 \) be an integer, and
let \( T \in \ogr{\mathcal{H}} \) be a
CPD operator such that \( T^n \) is
subnormal. Then \( T \) is subnormal.
   \end{theorem}
   \begin{proof}
Since subnormal operators are CPD,
$T^n$ is CPD. Denote by $(B, C, F)$ and
$(B_n, C_n, F_n)$ the representing
triples of $T$ and $T^n$, respectively,
and by $M$ and $M_n$ the semispectral
measures corresponding to $T$ and $T^n$
via Theorem~\ref{dyl-an}(ii),
respectively. Since $T^n$ is subnormal,
it follows from Theorem~\ref{subn-1}
that
   \begin{align} \label{zjnrk}
C_n = 0 \quad \text{and} \quad \int_{\rbb_+}
\frac{1}{(x - 1)^2} F_n(\D x) \Le I.
   \end{align}
By \cite[Theorem~3.3.2]{J-J-S22}, we
have
   \begin{align}  \label{mnbp}
M_n (\varDelta) = \tilde M_n
(\psi_n^{-1} (\varDelta)), \quad
\varDelta \in \borel{\rbb_+},
   \end{align}
where $\tilde M_n \colon \borel{\rbb_+}
\to \ogr{\hh}$ is the semispectral
measure defined by
   \begin{align} \label{tmnx}
\tilde M_n(\varDelta) =
\int_{\varDelta} (1+ x + \cdots +
x^{n-1})^2 M(\D x), \quad \varDelta \in
\borel{\rbb_+},
   \end{align}
and $\psi_n\colon \rbb_+\to \rbb_+$ is
given by
   \begin{align} \label{psnt}
\psi_n(x)=x^n, \quad x \in \rbb_+.
   \end{align}
This, together with
Theorem~\ref{dyl-an}, yields
   \begin{align} \notag
0 &\overset{\eqref{zjnrk}}= 2 C_n =
M_n(\{1\}) \overset{\eqref{mnbp}}=
\tilde{M}_n(\psi_n^{-1}(\{1\}))
   \\ \notag
&\hspace{1ex}= \tilde{M}_n(\{1\})
\overset{\eqref{tmnx}}= \int_{\{1\}}
(1+x+\dots+x^{n-1})^2 M(\D x)
=n^2M(\{1\}).
   \end{align}
This implies that $M(\{1\}) = M_n(\{1\}) = \tilde
M_n(\{1\}) = 0$, which, together with
\eqref{f2m-semi} applied to $T$ and $T^n$, shows
that $F = M$ and $F_n=M_n$. Hence, we obtain
   \begin{align} \notag
\int_{\rbb_+} \frac{1}{(x - 1)^2}
\langle F(\D x)f,f\rangle &=
\int_{\rbb_+} \frac{1}{(x - 1)^2}
\langle M(\D x)f,f\rangle
   \\  \notag
&\hspace{-1ex}\overset{\eqref{tmnx}}=
\int_{\rbb_+} \frac{1}{(x -
1)^2(1+x+\dots+x^{n-1})^2} \langle
\tilde{M}_n(\D x)f,f\rangle
   \\  \notag
& =\int_{\rbb_+} \frac{1}{(x^n -
1)^2}\langle \tilde{M}_n(\D
x)f,f\rangle
   \\  \notag
& =\int_{\rbb_+} \frac{1}{(x - 1)^2}
\langle \tilde{M}_n\circ \psi_n^{-1}
(\D x)f,f\rangle
   \\ \label{ihpq}
& \hspace{-1ex} \overset{\eqref{mnbp}}
= \int_{\rbb_+} \frac{1}{(x - 1)^2}
\langle M_n(\D x)f,f\rangle
\overset{\eqref{zjnrk}} \Le \|f\|^2,
\quad f \in \hh.
   \end{align}
This implies that $\frac{1}{(x - 1)^2}
\in L^1(F)$ and
   \begin{align} \label{asd}
\int_{\rbb_+} \frac{1}{(x - 1)^2} \D
F(x) \Le I.
   \end{align}
Now, we intend to show that
   \begin{align} \label{bfrdx}
B=\int_{\rbb_+} \frac{1}{x - 1} F(\D
x).
   \end{align}
First, observe that by the
Cauchy-Schwartz inequality,
   \allowdisplaybreaks
   \begin{align*}
\int_{\rbb_+} \frac{1}{|x - 1|} \langle
F(\D x)f,f\rangle & \Le \sqrt{\langle
F(\rbb_+)f,f\rangle} \;\;
\sqrt{\int_{\rbb_+} \frac{1}{(x - 1)^2}
\langle F(\D x)f,f\rangle}
   \\
& \hspace{-1ex}\overset{\eqref{asd}}
\Le \|F(\rbb_+)\|^{1/2}\|f\|^2, \quad f
\in \hh.
   \end{align*}
This implies that $\frac{1}{|x - 1|} \in L^1(F)$,
and thus the right-hand side of \eqref{bfrdx} is
well defined. Applying Theorem~\ref{subn-1} to
$T^n$, we obtain
   \begin{align} \notag
B_n & = \int_{\rbb_+} \frac{1}{x - 1}
F_n(\D x) = \int_{\rbb_+} \frac{1}{x -
1} M_n(\D x)
   \\ \notag
& \hspace{-1ex} \overset{\eqref{mnbp}}=
\int_{\rbb_+} \frac{1}{x - 1}
(\tilde{M}_n \circ \psi_n^{-1})(\D x) =
\int_{\rbb_+} \frac{1}{x^n - 1}
\tilde{M}_n (\D x)
   \\ \notag
& \hspace{-1ex}\overset{\eqref{tmnx}} =
\int_{\rbb_+} \frac{(x^n - 1)^2}{(x - 1)^2(x^n -
1)} M(\D x) = \int_{\rbb_+} \frac{x^n - 1}{(x -
1)^2} M(\D x)
   \\ \label{dbxo}
& = \int_{\rbb_+} \frac{x^n - 1}{(x - 1)^2} F(\D
x).
   \end{align}
Since $M(\{1\})=0$, it follows from
\eqref{f2m-remi} that
   \begin{align} \label{cjwz}
C=0.
   \end{align}
Thus, by \eqref{cdr5}, we have
   \begin{align*}
T^{*n}T^n = I + n B + \int_{\rbb_+} Q_n
\, \D F.
   \end{align*}
Hence, applying \eqref{f2m-remi} to
$T^n$, we obtain
   \begin{align*}
B_n \overset{\eqref{zjnrk}}= B_{n} +
C_{n} = T^{*n}T^n - I = n B +
\int_{\rbb_+} Q_n \, \D F.
   \end{align*}
This implies that
   \begin{align*}
n B & = B_n - \int_{\mathbb{R}_+} Q_n \, \D F
\overset{\eqref{dbxo}} = \int_{\mathbb{R}_+}
\frac{x^n - 1}{(x-1)^2} \, F(\D x) -
\int_{\mathbb{R}_+} Q_n\, \D F
   \\
& \hspace{-1ex}\overset{\eqref{rnx-1}} =
\int_{\mathbb{R}_+} \frac{x^n - 1}{(x-1)^2} \,
F(\D x) - \int_{\mathbb{R}_+} \frac{x^n-1 - n
(x-1)}{(x-1)^2} \, F(\D x)
   \\
& = n \int_{\mathbb{R}_+} \frac{1}{x-1}
\, F(\D x).
   \end{align*}
Thus, we have proved that \eqref{bfrdx}
holds.

Combining \eqref{asd}, \eqref{bfrdx},
and \eqref{cjwz}, and using
Theorem~\ref{subn-1}, we conclude that
$T$ is subnormal. This completes the
proof.
   \end{proof}
   \begin{corollary}\label{thpc}
Let \( n \Ge 2 \) be an integer, and let \( T \in \ogr{\mathcal{H}}
\) be a CPD operator such that \( T^n \) is normal. Then \( T \) is
normal.
   \end{corollary}
   \begin{proof}
Since $T^n$ is subnormal, Theorem~\ref{thp1} implies that $T$ is
subnormal (hence hyponormal), and the conclusion follows from
\cite[Theorem~5]{sta62}.
   \end{proof}
The following is a direct consequence of Theorem~\ref{thp1}.
   \begin{corollary}\label{col1}
If $T \in \ogr{\mathcal{H}}$ is not subnormal and $T^n$ is
subnormal for some integer $n \Ge 2$, then $T$ is not CPD.
   \end{corollary}
   \section{\label{Sec.5}Quasinormality and $m$-isometricity of CPD $n$th roots}
   As we show below, the class of quasinormal operators enjoys the
same property as the class of subnormal operators, namely, that it is
closed under taking CPD $n$th roots. We present two proofs. The first
is a direct consequence of Theorem~\ref{thp1} together with Theorem~
\ref{twPS}, while the second relies again on Theorem~\ref{thp1} and
on the characterization of quasinormality provided in
Theorem~\ref{thp3}.
   \begin{theorem} \label{thp5}
   Let $n\Ge 2$ be an integer, an let $T \in \ogr{\hh}$ be a CPD
operator such that $T^n$ is quasinormal. Then $T$ is quasinormal.
   \end{theorem}
   \begin{proof}[Proof I of Theorem~\ref{thp5}]
Since $T^n$ is quasinormal, it is
subnormal. Hence, by
Theorem~\ref{thp1}, $T$ is subnormal,
and consequently, by
Theorem~\ref{twPS}, $T$ is quasinormal.
   \end{proof}
\begin{proof}[Proof II of Theorem~\ref{thp5}]
As shown above, we deduce that $T$ is
subnormal. Since subnormal operators
are CPD, it follows that $T^n$ is CPD.
Denote by $(B, C, F)$ and $(B_n, C_n,
F_n)$ the representing triples of $T$
and $T^n$, respectively, and by $M$ and
$M_n$ the corresponding semispectral
measures as in
Theorem~\ref{dyl-an}(ii). It follows
from Theorem~\ref{subn-1} that
   \begin{align} \label{wltq}
\frac{1}{(x-1)^2} \in {L}^1({F}), \;\;
\frac{1}{x-1} \in {L}^1({F}), \;\; B =
\int_{\mathbb{R}_+} \frac{1}{x-1} \;\;
F(\D x) \;\; \text{and} \;\; C = 0.
   \end{align}
Applying \eqref{feme} to $T$ and $T^n$, we deduce
that $F=M$ and $F_n=M_n$. The measure transport
theorem then implies that (cf.\ \eqref{ihpq})
   \allowdisplaybreaks
   \begin{align}  \notag
\int_\varDelta \frac{1}{(x-1)^2} F(\D x) & =
\int_\varDelta \frac{1}{(x-1)^2} M( \D x )
   \\ \notag
&\hspace{-1ex}\overset{\eqref{tmnx}}{=}
\int_{\mathbb{R}_+}
\frac{\chi_\varDelta(x)}{(x-1)^2
(1+x+\cdots+x^{n-1})^2} \tilde{M}_n(\D x)
   \\  \notag
& =\int_{\mathbb{R}_+}
\frac{\chi_\varDelta(\psi_n^{-1}(\psi_n(x))
}{(\psi_n(x)-1)^2} \tilde{M}_n(\D x)
   \\  \notag
& =\int_{\mathbb{R}_+}
\frac{\chi_{\varDelta}(\psi_n^{-1}(x)) }
{(x-1)^2} \tilde{M}_n \circ\psi_n^{-1}(\D x)
   \\ \notag
& \hspace{-1ex} \overset{\eqref{mnbp}}{=}
\int_{\psi_n(\varDelta)}
\frac{1}{(x-1)^2}{M}_n(\D x)
   \\  \label{rvsw}
& = \int_{\psi_n(\varDelta)}
\frac{1}{(x-1)^2}{F}_n(\D x), \quad \varDelta \in
\borel{\rbb_+},
   \end{align}
where $\psi_n$ is as in \eqref{psnt}.
Since $T^n$ is quasinormal, it follows
from Theorem~\ref{thp3} that, for every
$\varDelta \in \borel{\rbb_+}$, the
operator $\int_{\psi_n(\varDelta)}
\frac{1}{(x-1)^2}{F}_n(\D x)$ is an
orthogonal projection. Hence, by
\eqref{rvsw}, for every $\varDelta \in
\borel{\rbb_+}$, the operator
$\int_{\varDelta} \frac{1}{(x-1)^2}F(\D
x)$ is an orthogonal projection.
Combined with \eqref{wltq} and
Theorem~\ref{thp3}, this implies that
$T$ is quasinormal.
   \end{proof}
Since subnormal operators are
automatically CPD, Theorem~\ref{twPS}
is a direct consequence of
Theorem~\ref{thp5}. Hence, the second
proof of the latter provides a further
justification of the former.

The class of $3$-isometries is yet another class that
enjoys the property of being closed under taking CPD $n$th
roots.
   \begin{theorem} \label{thp52}
Let $m,n \Ge 2$ be integers, and let $T
\in \ogr{\mathcal{H}}$ be a CPD
operator such that $T^n$ is an
$m$-isometry. Then $T$ is a
$3$-isometry.
   \end{theorem}
   \begin{proof}
    By \cite[Proposition~3.3.2]{J-J-S22}, the operator
$T^n$ is CPD and, by assumption, an $m$-isometry. It
follows from \cite[{Proposition~4.3.1}]{J-J-S22} that $T^n$
is a $3$-isometry and that $\mathrm{supp}(M_n) \subseteq
\{1\}$, where $M_n$ is the semispectral measure associated
with $T^n$ via Theorem~\ref{dyl-an}. On the other hand,
again by \cite[{Proposition~3.3.2}]{J-J-S22}, we have
    \begin{align*}
0 &= M_n(\mathbb{R}_+ \setminus \{1\})
\overset{\eqref{mnbp}}{=}
\tilde{M}_n(\psi_n^{-1}(\mathbb{R}_+
\setminus \{1\}))
=\tilde{M}_n(\mathbb{R}_+ \setminus
\{1\})
   \\
&\hspace{-1ex}\overset{\eqref{tmnx}}{=}
\int_{\mathbb{R}_+ \setminus \{1\}}
(1+x+\cdots+x^{n-1})^2 M(\D x),
  \end{align*}
where $M$ is the semispectral measure associated with $T$
via Theorem~\ref{dyl-an}, and $\psi_n$ is defined in
\eqref{psnt}. Consequently, $M(\mathbb{R}_+ \setminus
\{1\}) = 0$, which is equivalent to $\mathrm{supp}(M)
\subseteq \{1\}$. Applying
\cite[{Proposition~4.3.1}]{J-J-S22} once more, we conclude
that $T$ is a $3$-isometry.
   \end{proof}
It is worth recalling that, similarly to the class of CPD operators,
the class of $m$-isometries is closed under taking arbitrary positive
powers (see \cite[Theorem~2.3]{Ja02}). Moreover, every CPD
$m$-isometry is automatically a $3$-isometry.
   \section{\label{Sec.6}Examples}
Let $\EuScript{NMD}$ and $\EuScript{CPD}$ denote the classes of
normaloid operators and CPD operators, respectively. In this section,
we show that
   \begin{align}  \label{2ryl}
\EuScript{NMD} \not\subseteq \EuScript{CPD} \quad \text{and} \quad
\EuScript{CPD} \not\subseteq \EuScript{NMD}.
   \end{align}
Recall that an operator $T \in \ogr{\hh}$ is called {\em normaloid}
if $r(T)=\|T\|$, where $r(T)$ denotes the spectral radius of $T$.

We begin by proving, using a purely theoretical argument, that the
first relation listed in \eqref{2ryl} holds when the class
$\EuScript{NMD}$ is replaced by any scalable class of Hilbert space
operators that properly contains the class of subnormal operators.
Recall that a class $\EuScript{C}$ of Hilbert space operators is
called {\em scalable} if $\alpha T \in \EuScript{C}$ for every
$\alpha \in (0,\infty)$. Among others, the classes of normal,
subnormal, hyponormal, paranormal, and normaloid operators are
scalable (see \cite{Fur01} for more examples). On the other hand, the
classes of unitary, $m$-isometric, $2$-hyperexpansive, and CPD
operators are not scalable (see \cite[Lemma~1.21]{Ag-St1},
\cite[Lemma~1]{Rich88}, and \cite[Corollary~3.4.6]{J-J-S22}).
   \begin{prop}
Let $\EuScript{C}$ be a scalable class of Hilbert space operators
that is not contained in the class of subnormal operators. Then
$\EuScript{C} \not\subseteq \EuScript{CPD}$.
   \end{prop}
   \begin{proof}
Suppose, to the contrary, that $\EuScript{C} \subseteq
\EuScript{CPD}$. Let $T \in \EuScript{C}$. Then $\alpha T \in
\EuScript{C}$ for all $\alpha \in (0,\infty)$, and thus $\alpha T \in
\EuScript{CPD}$ for all $\alpha \in (0,\infty)$. By
\cite[Corollary~3.4.6]{J-J-S22}, this implies that $T$ is subnormal.
Hence, $\EuScript{C}$ is contained in the class of subnormal
operators, a contradiction.
   \end{proof}
To complement the preceding theoretical argument, we now provide a
concrete example illustrating the first relation listed in
\eqref{2ryl}, based on an idea due to Stampfli (see
\cite[pp.~378--379]{sta66}).
   \begin{exa}
The following two facts are needed in this example.
   \begin{align}  \label{cduq}
   \begin{minipage}{69ex}
{\em For every $\alpha \in (0,\infty)$, a unilateral weighted shift
with weights $\{\alpha, 1, 1, \ldots\}$ is subnormal if and only if
$\alpha \Le 1$ $($see \cite{sta66}$)$}.
   \end{minipage}
   \\[1ex] \label{cdxq}
   \begin{minipage}{69ex}
{\em For every $\kappa \in \natu$ and every sequence
$\{\alpha_j\}_{j=0}^{\kappa} \subseteq (0,\infty)$ such that
$\alpha_j \neq 1$ for some $j \in \{1, \ldots, \kappa\}$, the
unilateral weighted shift with weights $\{\alpha_0, \ldots,
\alpha_\kappa, 1, 1, \ldots\}$ is not subnormal $($use
\cite[Theorem~6]{sta66}$)$}.
   \end{minipage}
   \end{align}

Let $\kappa \Ge 2$ be a fixed integer, and let $T \in \ogr{\ell^2}$
be the unilateral weighted shift with weights
$\{\lambda_n\}_{n=0}^{\infty} \subseteq (0,1]$ such that $\lambda_n =
1$ for every $n \Ge \kappa$. Note that $T$ is normaloid because, by
\cite[Problem~91]{hal82}, we have
   \begin{align*}
r(T) = \lim_{n \to \infty} \sup_{k \Ge 0} \left( \prod_{j=0}^{n-1}
\lambda_{k+j} \right)^{1/n} = 1 = \sup_{k \Ge 0}\lambda_k = \|T\|.
   \end{align*}

The next observation is that $T^n$ is subnormal for every $n \Ge
\kappa$. Indeed, if $n \in \natu$, then the operator $T^n$ is
unitarily equivalent to the orthogonal sum $\bigoplus_{r=0}^{n-1}
T_{n,r}$ of unilateral weighted shifts $T_{n,r}$ with weights
$\big\{\prod_{j=0}^{n-1}\lambda_{r+kn+j}\big\}_{k=0}^{\infty}$, where
$r \in \{0, \ldots, n-1\}$. This implies that if $n \Ge \kappa$, then
the weights of $T_{n,r}$ take the form
   \begin{align*}
\bigg\{\prod_{j=0}^{n-1}\lambda_{r+j}, 1, 1, \ldots\bigg\}.
   \end{align*}
Since $\prod_{j=0}^{n-1}\lambda_{r+j} \Le 1$, we deduce from
\eqref{cduq} that each of the operators $T_{n,0}$, \ldots,
$T_{n,n-1}$ is subnormal, and consequently $T^n$ is subnormal as
well.

Suppose additionally that $\lambda_j \neq 1$ for every $j \in \{1,
\ldots, \kappa-1\}$. Then $T^n$ is not subnormal for every $n \in
\{1, \ldots, \kappa-1\}$. Indeed, otherwise there exists $n \in \{1,
\ldots, \kappa-1\}$ such that $T^n$ is subnormal, which implies that
$T_{n,0}$ is subnormal. Then the weights of $T_{n,0}$ take the form
   \begin{align*}
\bigg\{\prod_{j=0}^{n-1}\lambda_{j}, \prod_{j=0}^{n-1}\lambda_{n+j},
\ldots, \prod_{j=0}^{n-1}\lambda_{(\kappa-1)n+j}, 1, 1,
\ldots\bigg\}.
   \end{align*}
Since $\prod_{j=0}^{n-1}\lambda_{n+j} \neq 1$, we deduce from
\eqref{cdxq} that $T_{n,0}$ is not subnormal, a contradiction. By
Corollary~\ref{col1}, $T$ is not CPD.

Under the above assumptions, the operator $T$ has the following
properties:
   \begin{enumerate}
   \item[$\bullet$] $T$ is normaloid,
   \item[$\bullet$] $T^n$ is subnormal for every  $n \Ge \kappa$,
   \item[$\bullet$] $T^n$ is not subnormal for every $n \in \{1, \ldots, \kappa-1\}$,
   \item[$\bullet$] $T$ is hyponormal or not, depending on whether the
sequence $\{\lambda_n\}_{n=0}^{\kappa - 1}$ is monotonically
increasing or not; both possibilities can be realized,
   \item[$\bullet$] $T$ is not CPD.
   \hfill $\diamondsuit$
   \end{enumerate}
   \end{exa}
The next example concerns the second relation listed in \eqref{2ryl}.
   \begin{exa}
Let $T\in \ogr{\hh}$ be a non-isometric $3$-isometry. Then, by
\cite[Proposition~4.3.1]{J-J-S22}, $T$ is CPD. We show that $T$ is
not normaloid. Indeed, suppose to the contrary that $T$ is normaloid.
Then $r(T)=\|T\|$, and by \cite[Lemma~1.21]{Ag-St1}, $r(T)=1$, hence
$\|T\|=1$. However, by \cite{Ag-St1} (see also
\cite[Theorem~4.4]{J-J-S2020}), there exists a vector $h\in \hh$ such
that $\|T^n h\|^2$ is a polynomial in $n$ of degree at least $1$ (and
at most $2$). Consequently,
\[
\lim_{n\to\infty} \|T^n h\|^2 = \infty,
\]
which contradicts the inequality $\|T\|\Le 1$.

We now provide a concrete example via a unilateral weighted shift.
Let $p \in \mathbb{C}[x]$ be a polynomial of degree at least $1$ and
at most $2$ such that $p(n) > 0$ for every $n \in \zbb_+$. Define the
sequence $\{\lambda_n\}_{n=0}^{\infty} \subseteq (0,\infty)$ by
   \begin{align*}
\lambda_n = \sqrt{\frac{p(n+1)}{p(n)}}, \quad n \in \zbb_+.
   \end{align*}
Then, by \cite[Theorem~2.1]{Ab-Le16} (see also
\cite[Proposition~A.2]{J-J-S2020}), the unilateral weighted shift $T$
with weights $\{\lambda_n\}_{n=0}^{\infty}$ is a non-isometric
$3$-isometry.
   \hfill $\diamondsuit$
   \end{exa}
% \appendix

   \bibliographystyle{amsalpha}

\begin{thebibliography}{99}
\bibitem{Ab-Le16} B. Abdullah, T. Le,
The structure of $m$-isometric weighted shift operators, {\em Oper.
Matrices} {\bf 10} (2016), 319-334.

\bibitem{Ando72}
T. Ando, Operators with a norm condition,
{\em Acta Sci. Math. $($Szeged\/$)$} {\bf
33} (1972), 169--178.

%\bibitem{Ag85}
%J. Agler, Hypercontractions and
%subnormality, {\em J. Operator Theory}
%{\bf 13} (1985), 203--217.

   \bibitem{Ag-St1} J. Agler, M. Stankus, $m$-isometric
transformations of Hilbert spaces, I, {\it Integr. Equ. Oper. Theory}
{\bf 21} (1995), 383-429.
   \bibitem{Ag-St2} J. Agler, M. Stankus, $m$-isometric
transformations of Hilbert spaces, II, {\it Integr. Equ. Oper.
Theory} {\bf 23} (1995), 1-48.
   \bibitem{Ag-St3} J. Agler, M. Stankus, $m$-isometric
transformations of Hilbert spaces, III, {\it Integr. Equ. Oper.
Theory} {\bf 24} (1996), 379-421.

   \bibitem{Ash00}
R. B. Ash, {\em Probability and measure theory},
Harcourt/Academic Press, Burlington, 2000.

\bibitem{B-C-R}
C. Berg, J. P. R. Christensen, P. Ressel,
{\em Harmonic analysis on semigroups},
Springer-Verlag, Berlin, 1984.

\bibitem{Bir-Sol87}
M. Sh. Birman, M. Z. Solomjak, {\em Spectral
Theory of selfadjoint operators in Hilbert
space}, D. Reidel Publishing Co., Dordrecht,
1987.

\bibitem{brow53}
A. Brown, On a class of operators, {\em
Proc. Amer. Math. Soc.} {\bf 4} (1953),
723--728.

\bibitem{Buchala25} M. Bucha{\l}a, $m$-isometric
weighted shifts with operator weights, {\em Bull. Malays. Math. Sci.
Soc.} {\bf 48} (2025), 71.

\bibitem{Cha-Sh15} S. Chavan, V. M. Sholapurkar, {\em Completely
monotone functions of finite order and Agler's conditions}, {Studia
Math.} {226} (2015), 229-258.

\bibitem{Con91}
J. B. Conway, {\em The theory of
subnormal operators}, Math. Surveys
Monographs, Amer. Math. Soc., Providence,
1991.

\bibitem{con87}
J. B. Conway, B. B. Morrel, Roots and
logarithms of bounded operators on
Hilbert space, {\em J. Funct. Anal.} {\bf
70} (1987), 171--193.

\bibitem{Curto20}
R. E. Curto, S. H. Lee, J. Yoon,
Quasinormality of powers of commuting
pairs of bounded operators, {\em J.
Funct. Anal.} {\bf 278} (2020), 108342.

\bibitem{CurtoPrasad2026}
R. E. Curto, T. Prasad, Classes of
operators related to subnormal
operators, {\em Rev. Real Acad. Cienc.
Exactas Fis. Nat. Ser. A-Mat.}
\textbf{120} (2026), 41.

% \bibitem{Dou66}
% R. G. Douglas, On majorization, factorization,
% and range inclusion of operators on Hilbert
% space, {\em Proc. Amer. Math. Soc.} {\bf 17}
% (1966), 413--415.

\bibitem{Dug93} B. P. Duggal,  On $n$th roots of normal
contractions, {\em Bull. London Math. Soc.} {\bf 25} (1993), 74--80.

\bibitem{Emb68} M. R. Embry, $n$th
roots of operators, {\em Proc. Amer. Math. Soc.}
{\bf 19} (1968), 63--68.

   \bibitem{Embry73}
M. R. Embry, A generalization of the
Halmos-Bram criterion for subnormality,
{\em Acta Sci. Math. $($Szeged$)$} {\bf
35} (1973), 61--64.

\bibitem{Fur01} T. Furuta, {\em Invitation to
linear operators}, Taylor \& Francis, Ltd.,
London, 2001.

% \bibitem{FIY98}
% T. Furuta, M. Ito, T. Yamazaki, A
% subclass of paranormal operators
% including class of log-hyponormal and
% several related classes, {\em Sci. Math.}
% {\bf 1} (1998), 389--403.

%\bibitem{Gil74} F. Gilfeather, Operator valued roots of
%abelian analytic functions, {\em Pacific J.
%Math.} {\bf 55} (1974), 127--148.

\bibitem{hal50}
P. R. Halmos, Normal dilations and
extensions of operators, {\em Summa
Brasil. Math.} {\bf 2} (1950), 125--134.

%\bibitem{hal70}
%P. R. Halmos, Ten problems in Hilbert
%space, {\em Bull. Amer. Math. Soc.} {\bf
%76} (1970), 887--933.

\bibitem{hal82}
P. R. Halmos, {\em A Hilbert space
problem book}, Springer-Verlag, New York
Inc., 1982.

\bibitem{hal54}
P. R. Halmos, G. Lumer, Square roots of operators II, {\em
Proc. Amer. Math. Soc.} {\bf 5} (1954), 589--595.

\bibitem{hal53}
P. R. Halmos, G. Lumer, J. J. Schaffer,
Square roots of operators, {\em Proc.
Amer. Math. Soc.} {\bf 4} (1953),
142--149.

% \bibitem{He51}
% E. Heinz, Beitr\"age zur St\"orungstheorie der
% Spektralzerlegung, {\em Math. Ann.} {\bf 123}
% (1951), 415--438.

   \bibitem{Il-Ku19}
D. Ili\v{s}evi\'{c}, B. Kuzma, On square roots of
isometries, {\em Linear Multilinear Algebra} {\bf
67} (2019), 1898--1921.

% \bibitem{Ist67} V. Istr\u{a}\c{t}escu, On some hyponormal
% operators, {\em Pacific J. Math.} {\bf
% 22} (1967), 413--417.

% \bibitem{Ito99}
% M. Ito, Several properties on class A including
% $p$-hyponormal and log-hyponormal operators,
% {\em Math. Inequal. Appl.} {\bf 2} (1999),
% 569--578.

% \bibitem{Ito02}
% M. Ito, On classes of operators
% generalizing class A and paranormality,
% {\em Sci. Math. Jpn.} {\bf 7} (2002),
% 353--363.
   \bibitem{Ja02} Z. J. Jab{\l}o\'{n}ski, Complete hyperexpansivity,
subnormality and inverted boundedness conditions, {\em Integr. Equ.
Oper. Theory} {\bf 44} (2002), 316-336.
% \bibitem{jabl14}
% Z. J. Jab{\l}o\'{n}ski, I. B. Jung, J.
% Stochel, Unbounded quasinormal operators
% revisited, {\em Integr. Equ. Oper.
% Theory} {\bf 79} (2014), 135--149.

\bibitem{J-J-S2020}
Z. J. Jab{\l}o{\'n}ski, I. B. Jung, J. Stochel, $m$-isometric
operators and their local properties, {\em Linear Algebra Appl.}
\textbf{596} (2020), 49--70.

  \bibitem{J-J-S22} Z. J. Jab{\l}o{\'n}ski, I. B. Jung, J. Stochel, Conditional positive definiteness in operator theory,
 Dissertationes Math. 578, 2022, pp. 64.

% \bibitem{Jib10}
% A. A. S. Jibril, On operators for which
% $T^{*2} T^2 = (T^*T)^2$, {\em Int. Math.
% Forum} {\bf 46} (2010), 2255--2262.

% \bibitem{Kauf83} W. E. Kaufman, Closed operators
% and pure contractions in Hilbert space, {\em
% Proc. Amer. Math. Soc.} {\bf 87} (1983), 83--87.

\bibitem{Keo84} G. E. Keough, Roots of invertibly
weighted shifts with finite defect, {\em
Proc. Amer. Math. Soc.} {\bf 91} (1984),
399--404.

\bibitem{Ki-Ko22} Y. Kim, E. Ko, Characterizations
of square roots of unitary weighted composition
operators on $H^2$, {\em Complex Anal. Oper. Theory}
{\bf 16:14} (2022), 22 pp.

   \bibitem{Lam76}
A. Lambert, Subnormality and weighted
shifts, {\em J. London Math. Soc.} {\bf
14} (1976), 476--480.

% \bibitem{Lo34}
% K. L\"owner, \"Uber monotone Matrixfunktionen,
% {\em Math. Z.} {\bf 38} (1934), 177--216.

\bibitem{M-P-R23}
J. Mashreghi, M. Ptak, W. T. Ross, The square roots of some classical
operators, {\em Studia Math.} {\bf 269} (2023), 83--106.

% \bibitem{pp16}
% P. Pietrzycki, The single equality $A^{
% *n}A^n=(A^* A)^n$ does not imply the
% quasinormality of weighted shifts on
% rootless directed trees, {\em J. Math.
% Anal. Appl} {\bf 435} (2016), 338--348.

% \bibitem{pp18}
% P. Pietrzycki, Reduced commutativity of
% moduli of operators, {\em Linear Algebra
% Appl.} {\bf 557} (2018), 375--402.

\bibitem{P-S21}
P. Pietrzycki, J. Stochel, Subnormal $n$th roots
of quasinormal operators are quasinormal, {\em
J. Funct. Anal.} {\bf 280} (2021), 109001.

% \bibitem{P-S22} P. Pietrzycki, J. Stochel, Corrigendum to
% ``Subnormal $n$th roots of quasinormal operators
% are quasinormal'' [J. Funct. Anal. 280 (2021)
% 109001], {\em J. Funct. Anal.} {\bf 282} (2022),
% 109260.
\bibitem{P-S23}P. Pietrzycki, J. Stochel, On $n$th roots of bounded and unbounded quasinormal operators,
 Ann. Mat. Pura. Appl. 202 (2023), 1313-1333.

\bibitem{P-S-24} P. Pietrzycki, J. Stochel, Hyperrigidity I:
singly generated commutative $C^*$-algebras, to appear in
\emph{Israel J. Math.}

% \bibitem{P-S22-b}
% P. Pietrzycki, J. Stochel, Two-moment
% characterization of spectral measures on the
% real line, submitted.

\bibitem{put57}
C. R. Putnam, On square roots of normal
operators, {\em Proc. Amer. Math. Soc.}
{\bf 8} (1957), 768--769.

\bibitem{Ra-Ros} H. Radjavi, P. Rosenthal,
On roots of normal operators, {\em J. Math.
Anal. Appl} {\bf 34} (1971), 653--664.

\bibitem{Rich88} S. Richter, Invariant subspaces of
the Dirichlet shift, {\em Jour. Reine Angew. Math.} {\bf 386} (1988),
205-220.

%\bibitem{Rud73}
%W. Rudin, {\em Functional analysis},
%McGraw-Hill Series in Higher Math.,
%McGraw-Hill Book Co., New York, 1973.

\bibitem{Rud87} W. Rudin, {\em Real and Complex Analysis},
McGraw-Hill, New York 1987.

\bibitem{Sch12}
K. Schm\"{u}dgen, {\em Unbounded
self-adjoint operators on Hilbert space,}
Graduate Texts in Mathematics, 265,
Springer, Dordrecht, 2012.

   \bibitem{Sch17} K. Schm\"{u}dgen, {\em The moment
problem}, Graduate Texts in Mathematics
277, Springer, 2017.

\bibitem{Sh-At2000}
V. M. Sholapurkar, A. Athavale, Completely and alternatingly
hyperexpansive operators, {\em J. Oper. Theory} \textbf{43} (2000),
43--68.

\bibitem{sta62}
J. G. Stampfli, Hyponormal operators, {\em
Pacific J. Math.} {\bf 12} (1962), 1453--1458.

\bibitem{sta66}
J. G. Stampfli, Which weighted shifts
are subnormal?, {\em Pacific J. Math.}
{\bf 17} (1966), 367--379.

\bibitem{Stan23} H. Stankovi{\'c}, Subnormal $n$th
roots of matricially and spherically
quasinormal pairs, {\em Filomat} {\bf
37} (2023) 5325-5331.

\bibitem{Stankovic2025}
H. Stankovi\'{c}, Spherically
quasinormal tuples: $n$-th root problem
and hereditary properties, {\em Complex
Anal. Oper. Theory}, \textbf{19}
(2025), Article 156.

\bibitem{StankovicKubrusly2025}
H. Stankovi\'{c}, C. Kubrusly, On roots
of normal operators and extensions of
Ando's theorem, {\em Ann. Funct. Anal.}
\textbf{16} (2025), Article 60.

\bibitem{Sto92}
J. Stochel, Decomposition and
disintegration of positive definite
kernels on convex $*$-semigroups, {\em
Ann. Polon. Math.} {\bf 56} (1992),
243--294.

% \bibitem{Sto96}
%J. Stochel, Seminormality of operators
%from their tensor product, {\em Proc.
%Amer. Math. Soc.} {\bf 124} (1996),
%135--140.

% \bibitem{Sto02}
% J. Stochel, Lifting strong commutants of unbounded
% subnormal operators, {\em Integr. Equ. Oper. Theory}
% {\bf 43} (2002), 189--214.

% \bibitem{Sto-Szaf85} J. Stochel, F. H. Szafraniec,
% On normal extensions of unbounded operators. I,
% {\em J. Operator Theory} {\bf 14} (1985),
% 31--55.
   \bibitem{Sto-Szaf89} J. Stochel, F. H. Szafraniec,
On normal extensions of unbounded operators. II, {\em Acta Sci. Math.
$($Szeged$)$}, {\bf 53} (1989), 153--177.
%    \bibitem{Sto-Szaf89III}
% J. Stochel, F. H. Szafraniec, On normal
% extensions of unbounded operators. III. Spectral
% properties, {\em Publ. RIMS, Kyoto Univ.} {\bf
% 25} (1989), 105--139.
%    \bibitem{Sto-Szaf98}
% J. Stochel, F. H. Szafraniec, The complex moment
% problem and subnormality: a polar decomposition
% approach, {\em J. Funct. Anal.} {\bf 159}
% (1998), 432--491.
%   \bibitem{Sza00} F. H. Szafraniec, Subnormality in
%the quantum harmonic oscillator, {\em Comm.
%Math. Phys.} {\bf 210} (2000), 323--334.
%    \bibitem{Tana99} K. Tanahashi, On log-hyponormal
% operators, {\em Integr. Equ. Oper. Theory} {\bf
% 34} (1999), 364--372.
%    \Bibitem{uch93}
% M. Uchiyama, Operators which have
% commutative polar decompositions, {\em
% Oper. Theory Adv. Appl.} {\bf 62} (1993),
% 197--208.
%    \bibitem{uch01}
% M. Uchiyama, Inequalities for semibounded
% operators and their applications to
% log-hyponormal operators, {\em Oper. Theory Adv.
% Appl.} {\bf 127} (2001), 599--611.
   \bibitem{Weid80}
J. Weidmann, {\it Linear operators in
Hilbert spaces}, Springer-Verlag, Berlin,
Heidelberg, New York, 1980.
   \bibitem{wog85}
W. Wogen, Subnormal roots of subnormal
operators, {\em Integr. Equ. Oper.
Theory} {\bf 8} (1985), 432--436.
%    \bibitem{Yama99} T. Yamazaki,
% Extensions of the results on $p$-hyponormal and
% $\log$-hyponormal operators by Aluthge and Wang,
% {\em SUT J. Math.} {\bf 35} (1999), 139--148.

\end{thebibliography}
   
   \end{document}